\documentclass[10pt, twocolumn, draftcls]{ieeecolor}
\usepackage[utf8]{inputenc}
\usepackage{generic}
\usepackage{times,epsfig,graphicx,wrapfig,bbm,color}
\usepackage{bm, amssymb, amsmath}

\usepackage{color}
\definecolor{red}{rgb}{1,0.2,0.2}
\definecolor{green}{rgb}{0.2,1,0.5}
\definecolor{blue}{rgb}{0,0,1}
\definecolor{lightblue}{rgb}{0.3,0.5,1}
\definecolor{black}{rgb}{0,0,0}

\usepackage{hyperref}

\newcommand {\N} {{\rm I\kern-2.5pt N}}
\newcommand {\R} {{\rm I\kern-2.5pt R}}

\usepackage{amsthm}
\newtheorem{lemma}{Lemma}

\newtheorem{coro}{Corollary}
\newtheorem{theorem}{Theorem}
\newtheorem{assm}{Assumption}
\newtheorem{remark}{Remark}

\newcommand{\nnb}{\nonumber}
\newcommand{\1}{\mathbf{1}}
\newcommand{\0}{\mathbf{0}}

\newcommand{\dist}{\mathsf{d}}
\newcommand{\eb}{\mathbf{e}}

\newcommand{\gb}{\mathbf{g}}

\newcommand{\qb}{\mathbf{q}}

\newcommand{\T}{^{\mathsf{T}}}
\newcommand{\ub}{\mathbf{u}}

\newcommand{\wb}{\mathbf{w}}

\newcommand{\yb}{\mathbf{y}}
\newcommand{\zb}{\mathbf{z}}

\newcommand{\beqa}{\begin{eqnarray}}
\newcommand{\eeqa}{\end{eqnarray}}
\newcommand{\beqan}{\begin{eqnarray*}}
	\newcommand{\eeqan}{\end{eqnarray*}}
\newcommand{\beq}{\begin{equation}}
\newcommand{\eeq}{\end{equation}}

\newcommand{\bfl}{\begin{flushleft}}
	\newcommand{\efl}{\end{flushleft}}
\newcommand{\bgnv}{\begin{verbatim}}
	\newcommand{\ednv}{\end{verbatim}}

\newcommand{\myb}{\hspace{-0.1in}}
\newcommand{\myeq}{& \hspace{-0.1in} = & \hspace{-0.1in}}

\newcommand{\lb}{\nonumber \\}
\newcommand{\myarr}{\begin{array}{lll}}

	\newcommand{\cA}{{\cal A}}

	\newcommand{\bx}{{\bf x}}
	\newcommand{\by}{{\bf y}}
	
	\newcommand{\bv}{{\bf v}}

\newcommand{\btheta}{{\boldsymbol{\theta}}}

\newcommand{\bg}{{\bf g}}

\newcommand{\bV}{{\bf V}}

\newcommand{\myleq}{& \myb \leq & \myb}

\newcommand{\bitem}{\begin{itemize}}
	\newcommand{\eitem}{\end{itemize}}
\newcommand{\benum}{\begin{enumerate}}
	\newcommand{\eenum}{\end{enumerate}}

\newcommand{\norm}[1]{\left\| #1 \right\|}
\newcommand{\E}[1]{{\mathbbm E}\left[ #1 \right]}
\newcommand{\Es}[2]{{\mathbbm E}_{#1}\left[ #2 \right]}

\title{\LARGE \bf
Distributed Optimization with Global Constraints
Using Noisy Measurements
}

\author{Van Sy Mai, \ \and \ Richard J. La, \ \and \ 
    Tao Zhang, \ \and \ Abdella Battou\thanks{V.S. Mai, T. Zhang and A. Battou are 
    with the National Institute of Standards
    and Technology (NIST). R.J. La is with 
    NIST and the University of 
    Maryland, College Park.}
}

\begin{document}

\maketitle
\thispagestyle{empty}
\pagestyle{plain}

\begin{abstract}
We propose a new distributed optimization algorithm for
solving a class of constrained optimization problems in which
(a) the objective function is separable (i.e., the
sum of local objective functions of agents), (b) the
optimization variables of distributed agents, which 
are subject to nontrivial local constraints,  
are coupled by global constraints, and (c) only noisy 
observations are available to estimate (the gradients
of) local objective functions. In many practical 
scenarios, agents may not be willing
to share their optimization variables with others. 
For this reason, we propose a distributed algorithm that does
not require the agents to share their optimization 
variables with each other; instead, each agent 
maintains a local estimate of the global constraint 
functions and share the estimate only with its neighbors. 
These local estimates of constraint functions are updated
using a consensus-type algorithm, while the local
optimization variables of each agent are updated
using a first-order method based on noisy
estimates of gradient. We prove that, when the agents 
adopt the proposed algorithm, their optimization
variables converge with probability 1 
to an optimal point of an approximated problem based 
on the penalty method. 
\end{abstract}

\begin{keywords}
Distributed optimization, penalty method, potential game,
stochastic optimization.
\end{keywords}

\section{Introduction}
    \label{sec:Introduction}
    
Large engineered systems, such as telecommunication networks and electric power grids, are becoming
more complex and often comprise subsystems that use different technologies or are controlled autonomously or by different entities. Consequently, centralized resource management or system optimization becomes increasingly impractical or impossible. 
This naturally calls for a distributed optimization
framework that will enable distributed subsystems or agents to optimize both their local performance and, in the process, the overall system performance. 
What complicates the problem further is that many practical systems have global constraints (e.g., end-to-end delay requirements) that couple the decisions of more than one agent; hence satisfying these constraints will require coordination among the agents.

We study the problem of designing a distributed
algorithm for a set of agents to 
solve a constrained optimization problem. The 
problem has both (a) separate {\em local} constraints 
for each agent and (b) {\em global} constraints that 
agents must satisfy together and, hence, couple the 
optimization variables of the agents. Moreover, the 
analytic expression for the global objective function or 
its gradient is not assumed known to every agent. 
For example, in many practical cases, the objective functions or constraint functions are
summable (e.g., the aggregate cost of all agents). In such
cases, the local objective functions of an agent
may be unknown to other agents. Also, in many engineered
systems, the actual costs for optimizing resource utilization need to be estimated based on measurements, which contain observation noise. 
Consequently, the agents 
must rely on {\em noisy} observations to update their
optimization variables. 

This setting applies to a wide range of real-world applications, 
including communication networks. For example, in the fifth-generation (5G) and future 6G systems, many 
heterogeneous subsystems (e.g., radio access networks vs. wired core 
networks) wish to minimize their own costs of delivering services and, at the same time, need to collectively assure end-to-end
quality-of-service (QoS), e.g., end-to-end delays or 
packet loss rates, for different applications such as automated manufacturing, telesurgery, and remote controlled aerial and 
ground vehicles. 

Since the global constraint functions depend on the
optimization variables of more than one agent, 
their values are not always known to all agents. Thus, handling such global constraints will require the agents to exchange information. In many cases, computing and disseminating the exact
values of global constraint functions to all agents will incur excessive computing and communication overheads, or cause prohibitive delays. To cope with the challenge, 
we adopt the penalty method and propose a new 
consensus-based algorithm for agents to estimate the global 
constraint function values in a distributed manner. 
We prove that, under some technical conditions, 
the proposed algorithm ensures almost sure 
convergence to an optimal point when the 
problem with a suitable penalty 
function is convex.

\subsection{Related Literature}
    \label{subsec:Related}
    
Distributed optimization has attracted extensive attention and 
accumulated a large body of literature (e.g., \cite{Chang15, 
Cherukuri16, Nedic15, Nedic10, sundhar2010, sundhar2012new} and 
references therein), because it has broad applications and also 
serves as a foundation for distributed machine learning,
Here, we summarize only the most closely related studies
that considered stochastic constrained optimization 
problems with a summable objective function, and 
point out the key differences between them and our study. We 
emphasize that this is not meant to be an exhaustive list. 

Recently, consensus optimization has been an active
research area.
Srivastava and Nedi$\acute{\rm c}$~\cite{Srivastava11}
proposed a distributed stochastic optimization
algorithm for solving constrained optimization problems,
in which the feasible set is assumed to be the 
intersection of the feasible sets of individual agents. 
Each agent maintains a
local copy of the global optimization variables, which are updated
using a consensus-type algorithm. They showed that these
local copies converge asymptotically to a 
common optimal point with probability 1 (w.p.1). 
Bianchi and Jakubowicz~\cite{Bianchi13} proposed
a stochastic gradient algorithm for solving a
non-convex optimization problem. 
Every agent maintains a local copy of global optimization 
variables and updates its local copy using
a projected stochastic gradient algorithm that requires the knowledge of the global feasible set. 
The agents then exchange their local copies
with neighbors and update them using a consensus-type
algorithm. They proved that, under the proposed
algorithm, agents' local copies of optimization 
variables converge w.p.1 to the set of stationary
or Karush-Kuhn-Tucker (KKT) 
points of the global objective function.
In another study, Chatzipanagiotis and Zavlanos
\cite{Chatzi16} proposed a distributed algorithm
for solving a convex optimization problem 
with affine equality constraints
in the presence of noise. This algorithm is 
based on their earlier work called accelerated
distributed augmented Lagrangians 
\cite{Chatzi15, Chatzi17}, and requires exchanges of 
global optimization variables among all agents. 

We consider different settings in which the
objective functions are separable and local 
optimization variables are coupled via global constraints.
In order to cope with the coupling introduced by 
global constraints, each agent maintains local 
estimates of global constraint functions along with 
local optimization variables, and 
only the estimates of global constraint functions
are exchanged with neighbors. Our
algorithm is better suited for scenarios in which
agents may not want to share their own local 
optimization variables with each other. For instance, 
autonomous systems or domains with their own private networks 
may not wish to reveal how they manage their networks. Furthermore, a large system likely contains subsystems that utilize different 
networking technologies (e.g., WiFi networks, cellular 
radio access networks, and wired core networks), which
are managed differently. 
We also note that, as explained in \cite{Chang15}, the dual 
problem of the constrained optimization problem 
we consider in this paper can be formulated 
as a consensus optimization problem. However, 
existing algorithms and results in the literature are not 
applicable to our settings, especially with
noisy observations.

A more closely related line of research is distributed 
resource allocation. Many prior studies in this area focused 
on deterministic cases with linear resource constraints, 
where local objective functions are differentiable and strictly 
convex with Lipschitz continuous gradients that are available
to the agents (e.g.,~\cite{Chang15, Cherukuri16, 
nedic2018improved}). More recently, 
Yi et al.~\cite{yi2018distributed} considered stochastic
resource allocation problems with local constraints and
global {\em affine} equality constraints in the 
presence of both observation 
and communication noise, where local constraints are 
determined by continuously differentiable convex
functions. They proposed a new algorithm based 
on a primal-dual approach: each agent
maintains local multipliers and auxiliary variables
shared with other agents. 
These multipliers and auxiliary variables along with 
(primal) variables are updated using 
a combination of stochastic approximation (with 
decreasing step sizes) and consensus-type algorithms.

The distributed optimization problems 
we study can also be viewed as state-based 
games: Li and Marden 
\cite{Li14} formulated the problem of designing a distributed 
algorithm for solving an optimization problem with affine 
inequality constraints as one of designing 
{\em decoupled} utility functions for distributed
agents, using a penalty or barrier method.
This approach leads to a state-based potential 
game, whose potential function is the objective function of
an approximated problem, and the Nash equilibrium solves
the approximated problem. Consequently, under the assumption 
that the analytic
expressions for the objective functions are known 
and the exact gradients can be computed with no noise, 
they proved that when each agent tries to maximize 
its own local utility independently, their local decision
variables converge to an optimal point of the approximated 
problem.

{\bf {\em Our contributions:}}
All the above studies require exact local projections 
by assuming simple local constraint sets. In this paper, we 
relax this assumption and consider a distributed 
optimization problem where local objective functions are convex 
(possibly nonsmooth) and local constraints are not assumed to be
projection-friendly. We combine a stochastic subgradient 
method with a dynamic consensus tracking approach. Also, 
to deal with 
general local functional constraints, we adopt the idea of 
approximate projections from \cite{polyak2001random}; a similar 
approach is also used in \cite{lee2017approximate} for distributed 
consensus optimization.

In addition, our work extends the study 
by Li and Marden~\cite{Li14} to the case where the 
analytic expressions for the local objective functions of 
agents or their gradients are unavailable and instead 
need to be estimated using noisy observations. 
We demonstrate that the agents can 
track the global constraint functions using a simpler 
dynamic consensus tracking algorithm, even for
nonlinear global constraint functions. We prove that our 
algorithm converges to a Nash equilibrium of the aforementioned 
state-based potential game in \cite{Li14}, and no agent can 
increase its utility (equivalently, decrease its cost) in the 
state-based potential game via unilateral deviation.

{\bf Notation:} Define $\N := \{0, 1, 2
\ldots \}$ to be set of nonnegative integers.  
We use $x^+$ to denote $\max(0, x)$. Unless stated 
otherwise, all vectors are column vectors
and $\norm{\cdot}$
denotes the $\ell_2$ norm, i.e., $\norm{\cdot} 
= \norm{\cdot}_2$. Given a vector $\bx$ or a vector
function ${\bf f}$, we denote the $k$th element 
by $x_k$ and $f_k$, respectively. The vector of
zeros (resp. ones) of appropriate dimension is 
denoted by $\0$ (resp. $\1$). 
For a convex function $f:\mathcal{D} \subseteq \mathbb{R}^n \to \mathbb{R}$, 
we use $\partial f(\bx)$ to denote a subgradient of $f$ 
at a point $\bx \in \mathcal{D}$. When it is clear from the 
context, we will also use 
$\partial f(\bx)$ to denote the subdifferential at $\bx$. 
Given a closed convex set ${\cal S}$ and a vector $\bx$, 
$P_{{\cal S}}(\bx)$ denotes the projection of $\bx$ onto
the set ${\cal S}$ and $\dist_{\mathcal{S}}(\bx) := \| \bx - P_{{\cal S}}(\bx) \|$ is the distance from $\bx$ to $\mathcal{S}$. 
If $\mathcal{S}$ is a finite set, $|\mathcal{S}|$ denotes its cardinality. 
Given a sequence of random 
variables (RVs) or vectors ${\bf R}(t)$, $t \in \N$, 
we use ${\bf R}^t$ to denote the collection 
$\{{\bf R}(\tau) : \tau \in \{0, \ldots, t\}\}$.
%

\section{Setup and Proposed Algorithm}
    \label{sec:Setup+Algorithm}


    
We are interested in solving a constrained optimization 
problem of the form given below using a distributed
algorithm, where
(a) the objective function is separable, and (b) the
global (inequality) constraints couple the optimization
variables of the agents:
\begin{subequations}    \label{eq:GameForm}
	\begin{eqnarray}\label{ProbCVX_Game}
    \mbox{min}_{ \by \in {\cal G} } 
	&& \sum_{i \in \cA} \phi_i(\by_i)  \\
	\mbox{subject to}  
	&& \gb(\yb) := \sum_{i \in \mathcal{A}} \gb_{i}(\yb_i) 
	    \le \0 \label{eqConstraint}\\
	&& \yb_i \in \mathcal{G}_i \quad \mbox{for all } i \in \mathcal{A}, \label{eqLocalConstraint}    
	\end{eqnarray}
\end{subequations}
where $\cA$ is the set of $N$ agents, $\phi_i$ and 
$\yb_i$ are the local objective function and the local 
optimization variables, respectively, of 
agent $i$, $\by = (\by_i : i \in \cA)$ is the
vector of optimization variables, 
${\cal G}_i$ is the local constraint set of agent $i$, ${\cal G} := \prod_{i \in \cA} 
{\cal G}_i$ is the feasible set, and 
$\gb = (g_1, \ldots, g_K)$ is the vector consisting
of $K$ global constraint functions. 

We assume ${\cal G}_i = {\cal G}_{i0} \bigcap
\tilde{\cal G}_i$,
where ${\cal G}_{i0}$ is a nonempty  convex 
set assumed to be 
projection friendly (e.g., box, simplex, ball constraints), 
and the set $\tilde{\cal G}_i$ is given by 
$\tilde{\cal G}_i = \{ \yb_i \ | \ 
c_{ik}(\yb_i) \leq 0, \ k \in \tilde{\cal K}_i \}$, 
where $c_{ik}$'s are local constraint functions, and 
$\tilde{\cal K}_i$ is the set of agent $i$'s 
local (inequality) constraints. We are
interested in scenarios in which the constraint set 
$\tilde{\cal K}_i$ is large or $\tilde{\cal G}_i$
is not projection friendly. 
Here, $\mathcal{G}_{i0}$ can be regarded as a hard constraint while $\tilde{\mathcal{G}}_i$ represents soft constraints.

Unlike in \cite{Chatzi16, Li14}, we do not 
assume that the constraint functions are affine. Instead, we 
assume that each agent $i$ knows 
its contributions to global constraint function given by 
$\gb_i$, but not $\gb_j$, $j \in \cA \setminus \{i\}$.
This assumption is reasonable in many practical problems 
with distributed decision makers.

A popular approach to constrained optimization is 
the penalty method, which adds a penalty to the
objective function when one or more constraints are
violated. Although a general penalty 
function can be used in our problem, we assume a 
specific penalty function to facilitate our exposition.
Consider the following approximated problem 
with a penalty function:
\begin{equation}    \label{ProbCVX_Appr}
\mbox{min}_{ \by \in {\cal G} } \quad 
	\sum_{i \in \cA} \phi_i(\by_i)  
	    + \frac{\mu}{2N} \sum_{k \in {\cal K}_G}
	        g^+_k(\yb)^2
    =: \Phi(\yb), 
\end{equation}
where $\mu>0$ is a penalty parameter, ${\cal K}_G$ is the 
set of $K$ global constraints, and the second term 
in $\Phi$ is the penalty function. 
The original 
problem in \eqref{eq:GameForm} is recovered by 
letting $\mu \to \infty$. 
Throughout the paper, we assume that $\mu$ is fixed 
and known to all the agents.
The function $\Phi$ is 
convex on ${\cal G}$ provided that $\phi_i$ and 
$\gb_i$, $i \in \cA$, are convex.


The usual gradient projection method does not 
lead to a distributed algorithm because the 
penalty function and its gradient are coupled. 
Furthermore, we are interested in scenarios
in which (a) the analytic expression of local 
objective function $\phi_i$
and its gradient are unknown to agent $i$ and
(b) only noisy measurements are available to 
estimate them. 
To address these issues, we propose a new 
algorithm that requires each agent $i$ to keep an 
estimate $\eb_i(t)$ of the constraint functions 
$\bg(\yb(t))$, which are updated via dynamic 
consensus tracking.
These estimates are used to approximate the 
gradient of the penalty function and are updated
using a consensus-type algorithm. 

\subsection{Proposed Algorithm}
    \label{subsec:Algorithm}

Every agent $i \in {\cal A}$ will first randomly 
choose its initial local optimization
variables $\yb_i(0))$ in $\mathcal{G}_{i0}$ and its
estimate of (average) global constraint functions
$\eb_i(0) = \gb_i(\yb_i(0))$ based
on its own contributions.
Subsequently, at each iteration, it updates them 
according to the following algorithm: 
\begin{subequations}     \label{algorithm} 
\beqa
\hspace{-.25in}\zb_i(t) 
	\!\myeq  P_{{\cal G}_{i0}} \Big( \yb_i(t) - \gamma_t \big( \qb_i(t) + \bV_i(t)\big) \Big)  \label{algorithm-a} \\
\hspace{-.25in}\yb_i(t\!+\!1) 
	\!\myeq P_{{\cal G}_{i0}} \Big( \zb_i(t) 
	- \beta_{it} \sum_{k\in \mathcal{K}_{it}}
	\frac{c^+_{ik}(\zb_i(t))}
	{\| {\bf d}_{ik} \|^2} {\bf d}_{ik} \Big) 
	\label{algorithm-b} \\
\hspace{-.25in}\eb_i(t\!+\!1) 
	\!\myeq \!\sum_{j \in \mathcal{A}} w_{ij}(t) \eb_j(t) 
	\!+\! \gb_i(\yb_i(t\!+\!1)) \!-\! \gb_i(\yb_i(t))\label{eq:update_e}
\eeqa
\end{subequations}
where 
\beqa
\qb_{i}(t)  
	\!\myeq \partial \phi_i(\yb_{i}(t)) +  \mu \sum_{k \in {\cal K}_G} \partial g_{ik}(\yb_i(t)) e^+_{ik}(t), 
	\label{eq:update_q}
\eeqa 
$\bV_{i}(t)$ are random vectors,
$\gamma_t$ is a step size at iteration $t$,\footnote{We assume that the agents adopt  a common step size sequence, for example, based on 
	a global clock.} 
$\mathcal{K}_{it} \subset \tilde{{\cal K}}_i$ in \eqref{algorithm-b} is a random set of indices selected by agent $i$  
independently at each iteration $t$,%
\footnote{Note that one could replace the set of constraints $c_{ik}(\yb_i) \le 0, \ k \in \tilde{\cal K}_i$ with one constraint $\bar{c}_{i}(\yb_i) \le 0$ where $\bar{c}_{i}(\yb) = \max_{k\in \tilde{\mathcal{K}_i}} c_{ik}(\yb_i)$. Computationally, this, however, requires evaluations of all the constraint functions, which may be impractical when $|\tilde{\mathcal{K}_i}|$ is large.
}
and ${\bf d}_{ik}$ is a subgradient of 
$c_{ik}$ at $\zb_{i}(t)$ when  
$c_{ik}(\zb_i(t)) >  0$ and ${\bf d}_{ik} = {\bf d} 
\neq \0$ otherwise. 
The value of ${\bf d}$ is unimportant
since the summand is equal to zero in this case.

For simplicity, we assume that ${\cal K}_{it}$
is chosen according to a uniform distribution ${\cal U}_i$
over the set ${\cal S}_i := \{ S \subset \tilde{{\cal K}}_i
\ | \ |S| = s_i\}$ for some $s_i \geq 1$,
and take any 
$\beta_{it}$ in some $[\underline{\beta}_i,\bar{\beta}_i] \subset 
(0,2/s_i)$ for all iterations $t$.
In practice, we may choose $s_i \ll |\tilde{{\cal K}}_i|$ 
to limit the number of constraints considered in 
\eqref{algorithm-b} (even $s_i =1$). 
We assume that agent $i$ uses the same weight $\beta_{it}$ for all chosen local constraints in \eqref{algorithm-b}. But, it can employ different weights for different constraints which, for example, depend on the
value of constraint functions. 
We assume that ${\cal G}_{i0}$ is simple enough 
so that $P_{{\cal G}_{i0}}(\cdot)$ can be evaluated efficiently.

Note that agent $i\in \mathcal{A}$ 
needs to exchange its estimate $\eb_i$ only with its 
neighbors reflected in the weight matrix 
$W(t) = [w_{ij}(t) : i, j \in \cA]$, which is allowed to be 
time-varying. This update rule is designed based on a 
dynamic consensus algorithm (e.g., \cite{kia2019tutorial}) in order
to track the average of the global constraint functions. 
As a result, we can view $\eb_i(t)$ as a local estimate 
of the average global constraint functions.

We note that, in our setting, the subgradient $\partial 
\phi_i(\yb_{i}(t))$ in~\eqref{eq:update_q} is 
unavailable to agent 
$i$; instead, only the second term in \eqref{eq:update_q} 
can be computed by agent $i$ and a noisy 
estimate of the subgradient
given by the sum $\partial \phi_i(\yb_{i}(t)) + \bV_i(t)$ is 
available to the agent for the update in \eqref{algorithm-a}.


The random projection step in \eqref{algorithm-b}
is borrowed from \cite{polyak2001random} (see also 
\cite{nedic2011random}), but we do not restrict to 
selecting only a single constraint at each iteration.
The idea is that if the constraint 
$c_{ik}$ is violated 
at $\zb_i(t)$ for some randomly selected $k$, i.e., 
$c_{ik}(\zb_i(t)) > 0$, in order to 
reduce this violation, the algorithm updates the 
optimization variables in the direction of $-{\bf d}_{ik}$ 
with a step size $\beta_{it} 
\frac{c_{ik}(\zb_i(t))}{\| {\bf d}_{ik} \|^2}$
according to \eqref{algorithm-b}. 
This projection step, though simple, is beneficial in dealing with nontrivial and computationally intractable constraints, including linear matrix inequalities and robust linear inequalities (e.g., \cite{polyak2001random}).

For each $t \in \N$, we define ${\cal F}_t$ to be the 
$\sigma$-field generated by 
$\{\by(0), \bV^{t-1}, (\mathcal{K}_{i}^{t-1} : i \in \cA)\}$, 
where $\mathcal{K}_i^{t-1} =\{\mathcal{K}_{is} : 
0 \leq s < t\}$ and $\mathcal{K}_{is}$, $s \in \N$, is
the random index set of the constraints chosen by agent 
$i$ in \eqref{algorithm-b} at iteration $s$.
Throughout the paper, we use $\Es{t}{\cdot}$ to 
denote the conditional expectation $\E{\cdot \ | \ 
{\cal F}_t}$.

\subsection{Assumptions}
    \label{subsec:Assumptions}

We impose the following assumptions on the optimization problem
in \eqref{ProbCVX_Appr}.

\begin{assm}    \label{assm_prob}
Problem \eqref{ProbCVX_Appr} satisfies the following:
\benum

\item[a.] The local constraint sets ${\cal G}_i$, $i \in \cA$, 
are nonempty and convex, and ${\cal G}_{i0}$
are closed and convex.

\item[b.] The functions $\phi_i$ and $g_{ik}$, 
$i \in \cA$ and $k \in {\cal K}_G$, are convex 
and $L$-Lipschitz continuous on ${\cal G}_{i0}$ for some $L>0$.
Moreover, $\Phi$ is $L$-Lipschitz continuous on 
$\prod_{i \in \cA} {\cal G}_{i0}$.\footnote{This
assumption holds, for example, when ${\cal G}_{i0}$, 
$i \in \cA$, are compact.}

\item[c.] The local constraint functions $c_{ik}$, $i \in \cA$ 
and $k \in \tilde{\mathcal{K}}_i$, 
are convex and also $L$-Lipschitz continuous on ${\cal G}_{i0}$. 
Moreover, there exists $C>0$ such that, for all $i \in \cA$, 
\begin{equation}
\dist^2_{\mathcal{G}_i}(\zb) \le C\, \Es{\mathcal{K}\sim \mathcal{U}_i}{\sum_{k\in\mathcal{K}}c^+_{ik}(\zb)^2}, \quad \zb \in \mathcal{G}_{i0},    \label{eqRegularityCond}
\end{equation}
where $\Es{\mathcal{K} \sim \mathcal{U}_i}{\cdot}$ denotes the expectation when the random set $\mathcal{K}$ 
is chosen according to ${\cal U}_i$
over the set ${\cal S}_i$. 
	
\item[d.] The optimal set of \eqref{ProbCVX_Appr}, denoted by 
${\cal Y}^*$, is nonempty. 
\eenum
\end{assm}


Here, without loss of generality we assume the same Lipschitz constant $L$. Note that we do not assume differentiability of any involved functions or compactness of constraint sets. Assumption~\ref{assm_prob}-c, in particular condition~\eqref{eqRegularityCond}, is known to be quite general and plays an important role in the convergence analysis of algorithms involving random projection in \eqref{algorithm-b}; see \cite{nedic2011random} for a further discussion on this assumption as well as sufficient conditions for~\eqref{eqRegularityCond} to hold.

We allow the distribution of the 
perturbation ${\bf V}(t) := (\bV_i(t): i \in \cA)$ 
to depend on the optimization
variables $\by(t)$. The distribution of ${\bf V}(t)$
when $\by(t) = \by$ is denoted by 
$\mu_\by$.

\begin{assm}    \label{assm_noise}
	The perturbation satisfies (i) $\int \bv \ \mu_{\yb}(d\bv) 
	= \0$ for all $\yb \in {\cal G}$ and (ii)
	$\sup_{\yb \in {\cal G}} \int \norm{ {\bf v} }^2 
	\mu_{\yb}(d{\bf v}) =: \nu < \infty$.   
\end{assm}

For each $t \in \N$, define ${\cal E}_t = \{(i, j) 
\in \cA \times \cA : w_{ij}(t) > 0 \}$.

\begin{assm} \label{assm_graph}
	There exists $Q \in \N$ such that, for all $k \geq 0$,
	the graph $(\mathcal{A}, \bigcup_{l=1}^Q\mathcal{E}_{k+l})$ 
	is strongly connected (i.e., $Q$-strongly connected). 
	In addition, $W(t)$ is doubly stochastic for all $t \in 
	\N$, and there exists $w_{\min} > 0$ such that, for all 
	$t \in \N$, all nonzero weights $w_{ij}(t)$ lie in 
	$[w_{\min}, 1]$.
\end{assm}

We will use below the following standard assumptions on the step sizes. 

\begin{assm}    \label{assm_stepsize}
	Steps sizes $\gamma_t$, $t \in \N$, are positive and  satisfy 
	(i) $\sum_{t \in \N} \gamma_t = \infty$ and 
	(ii) $\sum_{t \in \N} \gamma_t^2 < \infty$.
\end{assm}

\subsection{Preliminaries}
Let us first introduce some 
notation. For every $t \in \N$, define
$\eb(t) := (\eb_i(t) : i \in \cA)$, 
${\gb}(t) := (\gb_i(\yb_i(t)) : i \in \cA)$,
\begin{align}
\bar{\eb}(t) 
:= \frac{1}{N} \sum_{i \in \mathcal{A}} \eb_i(t), \
\mbox{ and } \ \bar{\gb}(t):= \frac{1}{N} \gb(\yb(t)).
\end{align}



The following result is a direct consequence of the assumption 
that weight matrices $W(t)$, $t \in \N$, are doubly stochastic. 
\begin{lemma}   \label{lemma_AverageError}
	Under Assumption~\ref{assm_graph}, the following holds:
	\begin{eqnarray} 
	\bar{\eb}(t) 
	=  \bar{\gb}(t), \quad t \in \N.
	\label{eqAverageError}
	\end{eqnarray} 
\end{lemma}

We will also use the following result on the mixing property 
of a sequence of doubly stochastic weight matrices
\cite[Theorem~4.2]{sundhar2012new}.

\begin{lemma}   \label{lemConsensusWithNoiseQConnected}
	Consider the following iterates for all $i \in \cA$.
	\begin{equation*}
	\btheta_i(t+1) = \textstyle\sum_{j \in \mathcal{A}} 
	w_{ij}(t) \btheta_j(t) 
	+ \boldsymbol{\varepsilon}_i(t+1), 
	\quad t \in \N,
	\end{equation*}
	where $\boldsymbol{\varepsilon}_i(t) \in \R^{K}$ is arbitrary,  
	and $W(t)$ satisfies Assumption~\ref{assm_graph}. 
	Then, there exist $C_1>0$ and $C_2>0$ such that
	\begin{align*}
	\| \btheta_i(t) \!-\! \bar{\btheta}(t) \| 
	\le & C_1 N^{-1}\sigma^{t} 
	+ N^{-1}\textstyle\sum_{j \in \cA} 
	\| \boldsymbol{\varepsilon}_j(t) \|
	+ \| \boldsymbol{\varepsilon}_i(t) \| \\
	& +  C_2 N^{-1} \textstyle\sum_{s=1}^{t-1}
	\sigma^{{t-s}} 
	\sum_{j \in \cA} \| \boldsymbol{\varepsilon}_j(s) \|,
	\end{align*}
	where $\bar{\btheta}(t) = \sum_{j \in \mathcal{A}}
	\btheta_j(t) / N$, and $\sigma = \big( 1 - 
	\frac{w_{\min}}{4 N^2} \big)^{1 / Q}$. 
\end{lemma}

The lemma states that, over a larger timescale, 
the $Q$-strong connectivity has a similar mixing 
effect as a static, strongly connected graph. 

The following results will be used frequently in our 
proofs. 
\begin{lemma}   \label{lemma:cauchy2}
	For any $a,b\in \mathbb{R}$ and any $\eta > 0$,  
	\begin{equation}
	2 a b \leq \eta a^2 + \eta^{-1} b^2. \label{cauchy}
	\end{equation}
\end{lemma}

\begin{lemma}\label{lemBound_P}
Under Assumptions~\ref{assm_prob} and \ref{assm_graph}, for any $i\in \mathcal{A}$, 
\begin{align}
\|\qb_i(t)\|^2 
    \le 2L^2 \big( 1 + \mu^2 K \|\eb_i(t) - \bar{\eb}(t) \|^2 \big), 
    \quad t\in \N. \label{eqBound_Q}
\end{align}
\end{lemma}
\begin{proof}
We drop the dependence on $t$ for notational simplicity. It follows from~\eqref{eq:update_q} and \eqref{eqAverageError} that
\begin{align*}
\|\qb_i\| 
&= \| \partial_{\yb_i} \Phi(\yb) 
    + \mu \sum_{k \in {\cal K}_G} \partial g_{ik}(\yb_i)
    (e^+_{ik} - \bar{e}^+_k) \| \\
&\le  \|\partial_{\yb_i} \Phi(\yb)\| 
    + \mu \sum_{k \in {\cal K}_G} \|\partial g_{ik}(\yb_i)\|
    \cdot |e^+_{ik} - \bar{e}^+_k| \\
&\le L(1 + \mu\|\eb_i(t) - \bar{\eb}(t)\|_1), 
\end{align*}  
where we used the Lipschitz continuity of $\Phi, g_{ik}$ and the function $\max(\cdot, 0)$. 
Thus, 
\beqa
\|\qb_i(t)\|^2 
\myleq 2L^2 \big( 1 + \mu^2 \|\eb_i(t) - \bar{\eb}(t) \|_1^2 \big)
    \lb
\myleq 2L^2 \big( 1 + \mu^2 K \|\eb_i(t) - \bar{\eb}(t) \|^2 \big)
    \label{eq:qi_lowerbound}
\eeqa
which follows from the equivalence of norms.
\end{proof}

We also make use of the following classical results on the 
convergence of a sequence of nonnegative RVs. 
\begin{lemma} \cite[Lemma~10, p. 49]{polyak1987introduction} 
    \label{lemConvergenceRandomSeq}
Let $\{v_t : t \in \N\}$ be a sequence of nonnegative 
RVs with $\E{ v_0 } < \infty$. Suppose 
$\{\alpha_t : t \in \N \}$ and $\{\beta_t : t \in \N \}$ 
are deterministic scalar sequences satisfying
$\alpha_t \in [0,1]$, $\beta_t > 0$, 
$\sum_{t \in \N} \alpha_t = \infty$, 
$\sum_{t \in \N} \beta_t < \infty$,
$\lim_{t\to \infty} \beta_t/\alpha_t = 0$,
and
\beqan
& \E{ v_{t+1} | v^t } 
\leq (1-\alpha_t) v_t + \beta_t \ \mbox{ w.p.1 
    for all } t \in \N,  & 
\eeqan
where $v^t = \{v_s: 0 \leq s \leq t\}$.
Then, $v_t$ converges to 0 w.p.1, and $\lim_{t\to\infty} 
\E{ v_t } = 0$.
\end{lemma}

\begin{lemma} \cite[Lemma~11, p. 50]{polyak1987introduction} \label{lemConvergenceRandomSeq2}
Let $\{v_t : t \in \N \}, \{u_t : t \in \N \},$ 
$\{\alpha_t : t \in \N \}$ and $\{\beta_t : t \in \N \}$ be 
sequences of nonnegative RVs satisfying the 
following conditions w.p.1: $\sum_{t \in \N} \alpha_t < 
\infty$, $\sum_{t \in \N} \beta_t < \infty$, and 
\beqan
\E{ v_{t+1} | v^t, \alpha^t, \beta^t, u^t} 
\myleq (1 \! + \! \alpha_t)v_t \! - \! u_t \! + \! 
    \beta_t \ \mbox{ for all } t \in \N.
\eeqan
Then, w.p.1, (i) $\sum_{t \in \N} u_t < \infty$ and (ii) 
$v_t$ converges to some nonnegative RV $v$. 
\end{lemma}

\section{Convergence Analysis}   
    \label{sec:asynch}

In this section, we demonstrate that when all agents
update their local optimization variables ($\yb_i(t)$) 
and local estimates of global constraint functions ($\eb_i(t)$)
using our proposed algorithm 
in \eqref{algorithm}, $\yb(t)$ converges to the optimal 
set ${\cal Y}^*$ w.p.1. (Theorem~\ref{thm:main_convg}). 
We prove this
result with the help of a series of auxiliary results. 

First, the projection in \eqref{algorithm-b} satisfies the 
following inequality, the proof of which is inspired by \cite{lee2017approximate, nedic2011random, polyak2001random}. 

\begin{lemma} \label{lemmaSubgradProjection}
 Under Assumption~\ref{assm_prob},
for any $\yb_i \in {\cal G}_{i}$, we have
	\begin{subequations}
	\beqa
	&& \hspace{-0.4in} \| \yb_i(t\!+\!1) - \yb_i \|^2 
		\le  \| \zb_i(t) \!-\! \yb_i \|^2 
		- \tilde\beta_{it} \!\!\! \sum_{k\in \mathcal{K}_{it}}\!\! c^+_{ik}(\zb_i(t))^2 
		\label{eq:lemma7a} \\
	&& \hspace{-0.4in} c^+_{ik}(\zb_i(t))^2 
	\ge \textstyle \frac{\tau-1}{\tau} c^+_{ik}(\yb_i(t))^2 
	- 2\tau \gamma_{t}^2 L^2\big(2L^2 + \|\bV_i(t)\|^2 \big) 
	    \lb
	&& \hspace{0.55in} 
	- 4\tau \gamma_{t}^2 L^4\mu^2 K \|\eb_i(t) 
	- \bar{\eb}(t) \|^2
	    \label{eq:lemma7b}
	\eeqa
	\end{subequations}
	for all $i\in \mathcal{A}$, $k\in \mathcal{K}$,
	and $\tau > 0$, where $\tilde{\beta}_{it} = \frac{(2 - s_i\beta_{it})\beta_{it}} {L^2}$. 
\end{lemma}
\begin{proof}
The proof of \eqref{eq:lemma7a} is similar to those given 
in~\cite{nedic2011random, polyak2001random}, but we provide 
it here for completeness.
For any $\yb_i \in {\cal G}_{i}$, 
	\begin{align}
		& \| \yb_i(t\!+\!1) - \yb_i \|^2 
		\leq  \big\| \zb_i(t) \!-\! \yb_i - \textstyle \beta_{it}\sum_{k\in \mathcal{K}_{it}}  \frac{c^+_{ik}(\zb_i(t))}
		{\| {\bf d}_{ik} \|^2} {\bf d}_{ik} \big\|^2 \lb
		&= \|\zb_i(t) \!-\! \yb_i \|^2  
		+ \beta_{it}^2 \Big\|  \textstyle  \sum_{k\in \mathcal{K}_{it}}  \frac{c^+_{ik}(\zb_i(t))}{\| {\bf d}_{ik} \|^2}{\bf d}_{ik} \Big\|^2 \label{eq:lemma7-1} \\
		&\quad + 2 \textstyle \beta_{it}\sum_{k\in \mathcal{K}_{it}}  (\yb_i - \zb_i(t))^T{\bf d}_{ik}  \frac{c^+_{ik}(\zb_i(t))}
		{\| {\bf d}_{ik} \|^2}, \nonumber
	\end{align}
	where the first inequality follows from \eqref{algorithm-b} and the nonexpansiveness property of projection. Since $c^+_{ik}$ 
	is a convex function and ${\bf d}_{ik} \in \partial c^+_{ik}(\zb_i(t))$, 
	it follows that 
	$$(\yb_i \! - \! \zb_i(t) )^T{\bf d}_{ik} 
	\le c^+_{ik}(\yb_i) \! - \! c^+_{ik}(\zb_i(t)) = - c^+_{ik}(\zb_i(t)),$$ 
	where we used $c^+_{ik}(\yb_i) = 0$ for all $\yb_i 
	\in {\cal G}_{i}$.
	Moreover, by Cauchy-Schwarz inequality
	\begin{align*}
	\textstyle\Big\| \sum_{k\in \mathcal{K}_{it}}  \frac{c^+_{ik}(\zb_i(t))}{\| {\bf d}_{ik} \|^2}{\bf d}_{ik} \Big\|^2 
	&\le s_i \textstyle\sum_{k\in \mathcal{K}_{it}} \frac{c^+_{ik}(\zb_i(t))^2}{\| {\bf d}_{ik} \|^2}. 
	\end{align*}
	As a result, using the above bounds in \eqref{eq:lemma7-1}, 
	\begin{align*}
		& \| \yb_i(t\!+\!1) - \yb_i \|^2 \\
		&\leq  \|\zb_i(t) \!-\! \yb_i \|^2  + (s_i\beta_{it}^2  - 2\beta_{it}) \textstyle \sum_{k\in \mathcal{K}_{it}} \frac{c^+_{ik}(\zb_i(t))^2}{\| {\bf d}_{ik} \|^2},
	\end{align*}
	and \eqref{eq:lemma7a} now follows from 
	Assumption~\ref{assm_prob}-c.
	
	For the second inequality in \eqref{eq:lemma7b}, note that
	\begin{align}
	& c^+_{ik}(\zb_i(t))^2 \!\ge\! 
	c^+_{ik}(\yb_i(t))^2 \!+\! 2(c^+_{ik}(\zb_i(t)) \!-\! c^+_{ik}(\yb_i(t)))c^+_{ik}(\yb_i(t)) \nnb\\
	&\ge 
	c^+_{ik}(\yb_i(t))^2 - 2|c^+_{ik}(\zb_i(t)) - c^+_{ik}(\yb_i(t))|c^+_{ik}(\yb_i(t)) \nnb \\
	&\ge
	\textstyle\frac{\tau-1}{\tau}c^+_{ik}(\yb_i(t))^2 - \tau|c^+_{ik}(\zb_i(t)) - c^+_{ik}(\yb_i(t))|^2 \label{eqBound_czy}
	\end{align}
	for any $\tau > 0$, where we used \eqref{cauchy}. 
	In addition, 
	\begin{align*}
	    \!&\!|c^+_{ik}(\zb_i(t)) - c^+_{ik}(\yb_i(t))|^2 \\
	    \!&\!\le L^2 \|\zb_i(t) - \yb_i(t)\|^2 \tag{Lipschitz continuity}\\
	    \!&\!\le L^2\gamma_t^2\|\qb_i(t) + \bV_i(t)\|^2 \tag{projection property}\\
	    \!&\!\le 2L^2\gamma_t^2(\|\qb_i(t)\|^2 + \|\bV_i(t)\|^2)\\
	    \!&\!\le 4L^2\gamma_t^2\big(L^2 ( 1 \!+\! \mu^2 K \|\eb_i(t) \!-\! \bar{\eb}(t) \|^2 ) \!+\!
	    \textstyle \frac{\|\bV_i(t)\|^2}{2}\big). \tag{by~\eqref{eqBound_Q}}
	\end{align*}
	Substituting this bound in \eqref{eqBound_czy} gives us
	\eqref{eq:lemma7b}. 
\end{proof}



Although our result is similar to those in \cite{lee2017approximate, nedic2011random}, it contains an extra term $\|\eb_i(t) - \bar{\eb}(t)\|^2$ in \eqref{eq:lemma7b}, which measures the disagreements among the agents' local estimates of the global constraint functions.  Roughly speaking, for the algorithm to converge, it is necessary that this term decays to $0$. Much of the challenge lies in properly 
bounding this term, which is not bounded a priori. 
To this end, let us define
\begin{align}
a_t := \textstyle\sum_{i\in \mathcal{A}} \|\eb_i(t) - \bar{\eb}(t)\|^2 \label{at}, \quad t \in \N.
\end{align} 

The following result relates $a_t$ with $\| \yb(t) \!-\! \wb\|^2$ and  $\Phi(\wb) \!-\! \Phi\big(P_{\mathcal{G}}(\yb(t)) \big)$ for any $\wb \in \mathcal{G}$ and any choice of step size $\gamma_t$. 

\begin{theorem}\label{thm_Eyt1w}
	Under Assumptions~\ref{assm_prob} and \ref{assm_noise},
	the following holds for any $t \in \N$, $\wb \in  \mathcal{G}$, and positive sequence 
	$\{\kappa_t : t \in \N \}$:
	\begin{align}
	\!& \Es{t}{\| \yb(t+1) - \wb\|^2} \nnb\\ 
	\!&\le\! (1\!+\!\kappa_t K)\| \yb(t) \!-\! \wb\|^2   \!+\! \gamma^2_{t}D_1 \!+\!  \gamma^2_t\big(\mu^2L^2 \kappa_t^{-1}  \!\!+\! D_2 \big) a_t \nnb\\
	& + 2\gamma_{t}\big[\Phi(\wb) \!-\! \Phi\big(P_{\mathcal{G}}(\yb(t)) \big) \big] \!-\! 2\rho\dist^2_{\mathcal{G}}(\yb(t)), \label{eq_Eyt1w}
	\end{align}
	where 
	$\rho =  \frac{\tilde{\beta}}{4C}$ with $\tilde{\beta} = \inf_{i,t}\tilde{\beta}_{it}$, 
	$D_1 = \frac{L^2}{\rho}  + 10(2L^2N + \nu)$ and $D_2 = 20L^2 K\mu^2$.
\end{theorem}


\begin{proof}
First,  for any $\wb_i \in \mathcal{G}_i$, by Lemma~\ref{lemmaSubgradProjection}, we have
\begin{align}
\!\!\! \| \yb_i(t \!+\! 1) \!-\! \wb_i \|^2 \!\le\!  \| \zb_i(t) \!-\! \wb_i \|^2 \!-\! \tilde{\beta}_{it} \textstyle \sum_{k\in \mathcal{K}_{it}}   c^+_{ik}(\zb_i(t)) ^2. \nnb
\end{align} 
By the nonexpansive property of projection, the first term on 
the RHS can be expanded further as 
\begin{align*}
&\| \zb_i(t) \!-\! \wb_i \|^2 \le \big\| \yb_i(t) -\gamma_t\big( \qb_i(t) 
+ {\bf V}_i(t) \big) - \wb_i \big\|^2  \\
&=  \| \yb_i(t) - \wb_i\|^2 + \gamma_t^2 \|\qb_i(t) + {\bf V}_i(t) \|^2 \\
&\qquad + 2\gamma_t\big( \qb_i(t) + {\bf V}_i(t) \big)\T 
(\wb_i - \yb_i(t)). 
\end{align*} 
Thus, we have
\begin{align} 
&\| \yb_i(t+1) - \wb_i \|^2 \nnb \\
&\le  \| \yb_i(t) - \wb_i\|^2 -  \tilde{\beta}_{it}\textstyle\sum_{k\in \mathcal{K}_{it}} c^+_{ik}(\zb_i(t)) ^2 \nnb\\
& +2\gamma_t\big( \partial_{\yb_i}\Phi(\yb(t)) 
+ {\bf V}_i(t) \big)\T (\wb_i - \yb_i(t))  \nnb\\
& +2\gamma_t\mu\! \textstyle\sum_{k \in {\cal K}_G}  
\big| e^+_{ik}(t) 
\!\!- \bar{e}^+_k(t)\big| \cdot \big| \partial g_{ik}(\yb_i(t))\T(\wb_i \!- \yb_i(t)) \big| \nnb\\ 
& + \gamma_t^2 \|\qb_i(t) + {\bf V}_i(t) \|^2. \label{eqProofConvergence_2}
\end{align} 
We bound the last two terms in \eqref{eqProofConvergence_2}
as follows. First, 
\begin{align}
&\gamma_t^2 \|\qb_i(t) + {\bf V}_i(t) \|^2 
\leq 2\gamma_t^2 \big( \|\qb_i(t)\|^2 + \|{\bf V}_i(t) \|^2\big)  
    \nnb\\
&\le  2\gamma_t^2 \big( 2L^2 \mu^2 K \|\eb_i(t) - \bar{\eb}(t) \|^2 
    + 2L^2+ \|{\bf V}_i(t) \|^2\big), 
    \label{eqBoundI1}
\end{align}
where the second inequality follows from Lemma~\ref{lemBound_P}. 
Second, 
\begin{align}
&2\gamma_t\mu\!  \textstyle\sum_{k \in {\cal K}_G}  
\big| e^+_{ik}(t) 
\!\!- \bar{e}^+_k(t)\big| \cdot 
\big| \partial g_{ik}(\yb_i(t))\T(\wb_i \!- \yb_i(t)) \big| \nnb\\
&\le  \!\textstyle\sum_{k \in {\cal K}_G}  
2\gamma_t\mu L \big| e_{ik}(t) 
\!\!- \bar{e}_k(t)\big| \| \wb_i - \yb_i(t)\| \nnb \\
&\le \!\textstyle \sum_{k \in {\cal K}_G} \big( \gamma^2_t\mu^2 L^2 \kappa_t^{-1}
| e_{ik}(t) 
- \bar{e}_k(t)|^2 + \kappa_t \| \wb_i - \yb_i(t)\|^2 \big) \nnb\\
&\le \!\! \gamma^2_t\mu^2 L^2 \kappa_t^{-1}
\|\eb_i(t) - \bar{\eb}(t) \|^2 + \kappa_t K\| \wb_i - \yb_i(t)\|^2 \label{eqBoundI2}
\end{align}
for any $\kappa_t>0$, where we used the Lipschitz continuity of 
$g_{ik}$ in the first inequality, and \eqref{cauchy} of
Lemma~\ref{lemma:cauchy2} in the second inequality. 

Using the bounds of \eqref{eqBoundI1} and \eqref{eqBoundI2}  in~\eqref{eqProofConvergence_2} and then summing over $i\in 
\mathcal{A}$ gives us
\begin{align}
& \| \yb(t+1) - \wb\|^2 \nnb\\ 
&\le (1+\kappa_t K)\| \yb(t) - \wb\|^2 - \textstyle \sum_{i \in \cA}  \tilde{\beta}_{it}\textstyle\sum_{k\in \mathcal{K}_{it}} c^+_{ik}(\zb_i(t))^2 \nnb\\
&\quad+ 4NL^2\gamma_t^2 + 2\gamma_t^2\|{\bf V}(t) \|^2  + \gamma^2_t\mu^2(L^2 \kappa_t^{-1}  + 4L^2 K) a_t \nnb\\
&\quad +2\gamma_t\big( \partial \Phi(\yb(t)) 
+ {\bf V}(t) \big)\T (\wb- \yb(t)). \label{eqBound_ywu}
\end{align}
Now, using the convexity of $\Phi$ and Lemma~\ref{lemma:cauchy2}, 
we have 
\begin{align*}
&2\gamma_{t}\partial \Phi(\yb(t)) \T (\wb- \yb(t)) 
\le 2 \gamma_{t}\big(\Phi(\wb) - \Phi(\yb(t)) \big) \nnb\\
& =  2\gamma_{t} \big( \Phi(\wb) - \Phi(\ub)   + \Phi(\ub) - \Phi(\yb(t)) \big), \quad \forall \ub\in \mathcal{G}  \\
& \le  2\gamma_{t}\big(\Phi(\wb) - \Phi(\ub) \big)  + 2\gamma_{t}L\|\ub - \yb(t)\| \nnb\\
& \le  2\gamma_{t}\big(\Phi(\wb) - \Phi(\ub) \big)  + \gamma^2_{t}L^2\rho^{-1} + \rho \|\ub - \yb(t)\|^2 
\end{align*}
for any $\rho >0$. Taking $\ub = P_{\mathcal{G}}(\yb(t))$ yields 
\begin{align*}
&2\gamma_{t}\partial \Phi(\yb(t)) \T (\wb- \yb(t)) \\
&\le 2\gamma_{t}\big[\Phi(\wb) - \Phi\big(P_{\mathcal{G}}(y(t)) \big) \big]  + \gamma^2_{t}L^2 \rho^{-1} + \rho \dist^2_{\mathcal{G}}(\yb(t)).
\end{align*}
Using this bound in \eqref{eqBound_ywu} and then taking  
conditional expectation with Assumption~\ref{assm_noise} in 
place gives us
\begin{align}
& \Es{t}{\| \yb(t+1) - \wb\|^2} \nnb\\ 
&\le (1+\kappa_t K)\| \yb(t) - \wb\|^2 \lb
& \quad - \textstyle \sum_{i \in \cA} \tilde{\beta}_{it}\Es{t}{\textstyle\sum_{k\in \mathcal{K}_{it}} 
    c^+_{ik}(\zb_i(t))^2} \nnb \\
&\quad + \gamma^2_{t}(L^2 \rho^{-1} + 4NL^2 + 2\nu)
    + \gamma^2_t\mu^2(L^2 \kappa_t^{-1}  + 4L^2 K) a_t \nnb\\
&\quad + 2\gamma_{t}\big[\Phi(\wb) - \Phi\big(P_{\mathcal{G}}(\yb(t)) \big) \big] + \rho \dist^2_{\mathcal{G}}(\yb(t)). \label{eqBound_ywu2}
\end{align}
We now bound the second term on the RHS. By Lemma~\ref{lemmaSubgradProjection}, for any $\tau > 0$,
\begin{align*} 
&\textstyle \sum_{i \in \cA}  \tilde{\beta}_{it}\Es{t}{\textstyle\sum_{k\in \mathcal{K}_{it}} 
    c^+_{ik}(\zb_i(t))^2}  \nnb\\
&\ge 
\textstyle \sum_{i \in \cA}  \tilde{\beta}_{it} \Es{\mathcal{K} \sim \mathcal{U}_i}{\textstyle\sum_{k\in \mathcal{K}} 
    \textstyle \frac{\tau-1}{\tau} c^+_{ik}(\yb_i(t))^2} \nnb\\
&\quad - \textstyle \sum_{i \in \cA} \tilde{\beta}_{it} s_i
	2\tau \gamma_{t}^2 L^2\big(2L^2 + \Es{t}{\|\bV_i(t)\|^2} \big)  \nnb\\
&\quad - \textstyle \sum_{i \in \cA} \tilde{\beta}_{it} s_i
    4\tau \gamma_{t}^2 L^4\mu^2 K \|\eb_i(t) 
	- \bar{\eb}(t) \|^2.
\end{align*}
Taking $\tau=4$, using Assumption~\mbox{\ref{assm_prob}-c}, 
$\tilde{\beta} = \inf_{i,t}\tilde{\beta}_{it} $ and the fact  that 
$\tilde{\beta}_{it} s_i  
= \frac{(2- \beta_{it} s_i)\beta_{it} s_i }{L^2} \le \frac{1} {L^2}$, we then have
\begin{align} 
&\textstyle \sum_{i \in \cA} \tilde{\beta}_{it}\Es{t}{\textstyle\sum_{k\in \mathcal{K}_{it}} 
    c^+_{ik}(\zb_i(t))^2} \nnb\\
&\ge \textstyle \textstyle \sum_{i \in \cA}  \frac{3\tilde{\beta}}{4 C} \dist^2_{\mathcal{G}_i}(\yb_i(t)) - 8 \gamma_{t}^2 \big(2NL^2 + \nu\big)  - 
16 \gamma_{t}^2 L^2\mu^2 K a_t \nnb\\
&\ge \textstyle \frac{3\tilde\beta}{4C}\textstyle \dist^2_{\mathcal{G}}(\yb(t))  - 8\gamma_{t}^2 \big(2NL^2 + \nu \big)  - 
16\gamma_{t}^2 L^2\mu^2 K a_t .
\label{eq:cik_lowerbound}
\end{align}
Using this bound in~\eqref{eqBound_ywu2} with $\rho =  \frac{\tilde{\beta}}{4C}$ and then rearranging terms yields~\eqref{eq_Eyt1w}. 
This completes the proof.
\end{proof}

Note that $\tilde{\beta} = \inf_{i,t} \tilde{\beta}_{it}$ is strictly positive since we assume that $\beta_{it} \in [\underline{\beta}_i,\bar{\beta}_i] \subset 
(0,2/s_i)$, and thus, $\tilde{\beta}_{it} \ge \frac{(2 - s_i\bar{\beta}_{i})\underline{\beta}_{i}} {L^2} >0$ for all $i\in\mathcal{A}$ and $t\in \N$.

Note also that we introduced a positive sequence $\{\kappa_t : t \in \N \}$ in~\eqref{eq_Eyt1w}, which will be chosen appropriately for various intermediate results below. The following is obtained by choosing a constant sequence. 
\begin{coro}\label{coroBound_Edy1}
	For all  $t \in \N$, we have 
	\begin{align}
 	&\Es{t}{\|\yb(t\!+\!1) - P_\mathcal{G}(\yb(t)) \|^2} \nnb\\
 	&\qquad \le \epsilon_d \dist^2_{\mathcal{G}}(\yb(t))  + \gamma^2_{t}D_1 + \gamma^2_t D_3 a_t, \label{eqBound_Edy1}
	\end{align} 
	where $\epsilon_d= 1 - \rho$, and $D_3=\mu^2L^2 K \rho^{-1}\!+\! D_2.$ 
\end{coro}
\begin{proof}
	From \eqref{eq_Eyt1w} in Theorem~\ref{thm_Eyt1w} with $\wb = P_{\mathcal{G}}(\yb(t))$, we have, for any $\kappa_t>0$ and $t \in \N$,
	\begin{align}
	& \Es{t}{\|\yb(t\!+\!1) \!-\! P_\mathcal{G}(\yb(t)) \|^2} \nnb\\
	& \le (1+\kappa_t K  -2\rho ) \dist^2_{\mathcal{G}}(\yb(t)) 
	 + \gamma^2_{t}D_1 \!+\! \gamma^2_t\big(\mu^2L^2 \kappa_t^{-1}  \!\!+D_2 \big) a_t.  \nnb
	\end{align}
	Thus, \eqref{eqBound_Edy1} follows by taking $\kappa_t \equiv \frac{\rho}{K}$ for all $t \in \N$.
\end{proof}

Since $\dist^2_{\mathcal{G}}(\yb(t+1))\le \|\yb(t+1)-P_{\mathcal{G}}(\yb(t))\|^2$, a direct consequence of this corollary is that
\begin{align}
 \Es{t}{\dist^2_{\mathcal{G}}(\yb(t+1))}  \le \epsilon_d \dist^2_{\mathcal{G}}(\yb(t))  + \gamma^2_{t}D_1 
 + \gamma^2_t D_3 a_t. \label{eqBound_Edy2}
\end{align} 
This result hints at a form of contraction property of the sequence $\{ \dist^2_{\mathcal{G}}(\yb(t)) : t \in \N \}$. If $\{a_t : t \in \N \}$ were  bounded, together with Assumption~\ref{assm_stepsize}, 
\eqref{eqBound_Edy2} would immediately imply that $\lim_{t \to \infty}
\dist^2_{\mathcal{G}}(\yb(t)) = 0$. However, as $a_t$ depends on $\|  \gb(\yb(t+1)) - \gb(\yb(t))\|$, which in turn depends on $\| \yb(t+1) - \yb(t)\|$, such convergence does not follow from~\eqref{eqBound_Edy2}. Below, we shed some light on this through a customized linear coupling argument.

\begin{theorem} \label{thmEdyba}
    Suppose that Assumptions~\ref{assm_prob}--\ref{assm_graph} hold and step sizes satisfy $\lim_{t \to \infty} \gamma_t = 0$. Then,  
	there exist $\epsilon, \epsilon_a, \epsilon_b \in (0,1)$ such that the following holds for all $t$ sufficiently large:
	\begin{align}
	&\Es{t}{\dist^2_{\mathcal{G}}(\yb(t+1)) + \epsilon_b b_{t+1} + \epsilon_a a_{t+1}} \nnb\\
	& \le 
	\epsilon \big( \dist^2_{\mathcal{G}}(\yb(t))  +  \epsilon_b b_t \big) + \gamma_t^2 a_t(1+\epsilon_b)D_3   +  \mathcal{O}(p_t), \label{eqEdyba}
	\end{align}
	where $D_3$ is given in Corollary~\ref{coroBound_Edy1}, 
	\begin{align}
	&b_t  = \textstyle\sum_{s=0}^t\sigma^{{t- s}} 
	\big( \dist^2_{\mathcal{G}}(\yb(s)) + \frac{1}{2}\gamma_s^2 a_s D_3 \big), \ \mbox{ and } \nnb\\
	&p_t = \gamma^2_{t} + 
	\sigma^{2(t+1)} + \textstyle\sum_{s=0}^t \sigma^{{t- s}} \gamma_s^2\nnb
	\end{align}
	with $\sigma \in (0,1)$ given in Lemma~\ref{lemConsensusWithNoiseQConnected}. 
\end{theorem}

\begin{proof}
	Applying Lemma~\ref{lemConsensusWithNoiseQConnected} with 
	$\btheta_i(t) = \eb_i(t)$ and $\boldsymbol{\varepsilon}_i(t+1) = 
	\Delta {\gb}_i(t) := {\gb}_i(\yb_i(t+1)) - {\gb}_i(\yb_i(t))$ and then summing over $i\in \cA$ yields
	\begin{align*}
	\sum_{i \in \cA}\| \eb_i(t) - \bar{\eb}(t) \| 
	\le & C_1\sigma^{t} + C_2\sum_{s=0}^{t-1} \sigma^{{t-1- s}} 
	\sum_{i \in \cA} \| \Delta {\gb}_i(s) \|.
	\end{align*}
	This relation and Cauchy-Schwartz inequality imply
	\begin{align}
	&\textstyle \frac{1}{2}a_{t+1} 
	\le 
	\textstyle \frac{1}{2}\big(\sum_{i \in \cA}\| \eb_i(t+1) - \bar{\eb}(t+1) \| \big)^2 \nnb\\
	& \le 
	C_1^2\sigma^{2(t+1)} + C_2^2 \big(\textstyle \sum_{s=0}^{t}\sigma^{{t- s}} 
	\sum_{i \in \cA} \| \Delta {\gb}_i(s) \|\big)^2 \nnb\\
	& \le 
	C_1^2\sigma^{2(t+1)} + \textstyle \frac{C_2^2N}{1-\sigma} \sum_{s=0}^{t}\sigma^{{t- s}} 
	\sum_{i \in \cA} \| \Delta {\gb}_i(s) \|^2.  \label{eqat_iterbound}
	\end{align} 
	For the last inequality, we make use of the 
	following inequalities: i) $(\sum_k h_k m_k)^2 \le (\sum_k h_k^2)(\sum_k m_k^2)$ with $h_k = \sqrt{\sigma^{{t - k}}}$ and 
	$m_k = \sqrt{\sigma^{{t- k}}} \sum_{i \in \cA} \| \Delta {\gb}_i(k) \|$, ii) $\sum_{s=0}^{t}\sigma^{{t- s}} \le \frac{1}{1-\sigma}$ for all $t \in \N$, and iii) $\big(\sum_{i \in \cA} \| \Delta {\gb}_i(s) \|\big)^2 \le N\sum_{i \in \cA} \| \Delta {\gb}_i(s) \|^2$.
	
	Next, we bound the last term in~\eqref{eqat_iterbound}. 
	Note that
	$\| \Delta {\gb}_{i}(s) \| \leq \sqrt{K} \|
	\Delta {\gb}_i(s) \|_{\infty}$. Thus, we have
	\begin{align}
	&\| \Delta {\gb}_i(s) \|^2 
	\le  K L^2 \|\yb_i(s+1) - \yb_i(s)\|^2 \nnb\\
	&\le  
	2K L^2 \big( \|\yb_i(s+1) - P_{\mathcal{G}_i}(\yb_i(s))\|^2 + \dist^2_{\mathcal{G}_i}(\yb_i(s)) \big) . \nnb
	\end{align} 
Therefore, 
\begin{align*}
&\Es{s}{\textstyle \sum_{i \in \cA} \| \Delta {\gb}_i(s) \|^2} \\
&\le  
2K L^2 \big(\Es{s}{ \|\yb(s+1) - P_{\mathcal{G}}(\yb(s))\|^2} + \dist^2_{\mathcal{G}}(\yb(s)) \big) \\
&\le 2K L^2 \big( 2\dist^2_{\mathcal{G}}(\yb(s))  + \gamma^2_{s}D_1 + \gamma^2_s a_s D_3 \big),
\end{align*} 
where the last inequality follows from~\eqref{eqBound_Edy1} and the fact that $\epsilon_d<1$. 
Using this bound in \eqref{eqat_iterbound} after taking the expectation, 
\begin{align}
\!\!\Es{t}{a_{t+1}} 
\!\le\! 
2\big(C_1^2\sigma^{2t+2} \!+\! C_3 b_t \!+\! \textstyle 
K L^2 D_1\sum_{s=0}^t \sigma^{{t- s}} \gamma_s^2\big) \label{eqBound_Eat1}
\end{align}
where 
$b_t = \sum_{s=0}^{t}\sigma^{{t- s}} 
\big( \dist^2_{\mathcal{G}}(\yb(s)) + \gamma_s^2 a_s\frac{D_3}{2} \big)$ and $C_3 = 4C_2^2N K L^2(1-\sigma)^{-1}$.  
Note that 
\begin{align}
b_{t+1} = \sigma b_t + \dist^2_{\mathcal{G}}(\yb(t+1)) + \gamma_{t+1}^2 a_{t+1} \textstyle \frac{D_3}{2}\nnb
\end{align}
which, when taking conditional expectation of both sides and using  \eqref{eqBound_Eat1}, implies
\begin{align}
\Es{t}{b_{t+1}} 
&\le 
(\sigma+\gamma_{t+1}^2C_3D_3 ) b_t + \Es{t}{\dist^2_{\mathcal{G}}(\yb(t+1))} \nnb\\
&\quad + \gamma_{t+1}^2 D_3\big(C_1^2\sigma^{2t+2} \!+\! \textstyle 
K L^2 D_1\sum_{s=0}^t \sigma^{{t- s}} \gamma_s^2 \big).\nnb
\end{align}
Then, in view of the bound in~\eqref{eqBound_Edy2}, we further have
\begin{align}
\Es{t}{b_{t+1}} 
&\le 
(\sigma+\gamma_{t+1}^2C_3D_3 ) b_t + \epsilon_d \dist^2_{\mathcal{G}}(\yb(t)) \nnb\\
&\quad  + \gamma^2_t a_t D_3 + \tilde{p}_t, \label{eqBound_Ebt1}
\end{align}
where $\tilde{p}_t \!=\! \gamma_{t+1}^2 D_3\big(C_1^2\sigma^{2t+2} \!+\! 
K L^2 D_1\!\sum_{s=0}^t \sigma^{{t- s}} \gamma_s^2 \big) \!+ \gamma^2_{t}D_1.$

We now couple \eqref{eqBound_Edy2}, \eqref{eqBound_Eat1} and \eqref{eqBound_Ebt1} as follows: 
for any $\epsilon_a>0$ and $\epsilon_b >0$, add \eqref{eqBound_Edy2}, \eqref{eqBound_Eat1} multiplied by $\epsilon_a$, and \eqref{eqBound_Ebt1} multiplied by $\epsilon_b$, and then simplify the expression.  
\begin{align*}
&\Es{t}{\dist^2_{\mathcal{G}}(\yb(t+1)) + \epsilon_b b_{t+1} + \epsilon_a a_{t+1}} \nnb\\
& \le 
\epsilon_d(\epsilon_b+1) \dist^2_{\mathcal{G}}(\yb(t))  + \big(\epsilon_b\sigma+\gamma_{t+1}^2\epsilon_b C_3{D_3} + 2\epsilon_aC_3\big)  b_t \nnb\\
&\quad + \gamma^2_t a_t(1+\epsilon_b)D_3  + \bar{p}_t, 
\end{align*}
where 
\begin{align}
    \bar{p}_t &=
    \epsilon_a 2\big( C_1^2\sigma^{2t+2} \!+\!
K L^2 D_1\textstyle\sum_{s=0}^t \sigma^{{t- s}} \gamma_s^2 \big) + \gamma^2_{t}D_1 + \epsilon_b\tilde{p}_t\nnb\\
    &=  \mathcal{O}\big(\gamma^2_{t} + 
	\sigma^{2(t+1)} + \textstyle\sum_{s=0}^t \sigma^{{t- s}} \gamma_s^2 \big).\label{eqPbar}
\end{align}
Let us now choose $\epsilon_a$ and $\epsilon_b$ sufficiently small so that there exists $\epsilon \in (0,1)$ satisfying
$$
	\epsilon_d(\epsilon_b+1) < \epsilon \ \mbox{ and } \  \epsilon_b(\sigma+\gamma_{t+1}^2C_3{D_3}) + 2\epsilon_aC_3 < \epsilon \epsilon_b
$$
for sufficiently large $t$ (with $\gamma_{t+1}^2 < \frac{1}{C_3 D_3}(\epsilon - \sigma - \epsilon_a\frac{2C_3}{\epsilon_b})$). 
Thus, 
we conclude that~\eqref{eqEdyba} holds 
for sufficiently large $t$. 
\end{proof} 	

Note that the exact form of the term  $\mathcal{O}(p_t)$, given by $\bar{p}_t$ in~\eqref{eqPbar}, is not important for us, as we will see later that this term behaves like $\gamma_t^2$ and, thus, is summable under Assumption~\ref{assm_stepsize}. 

It is now clear that the sequence $\{ \dist^2_{\mathcal{G}}(\yb(t)) + \epsilon_b b_{t} + \epsilon_a a_{t} : t\in\N \}$ possesses a contraction property. As a result, we have the following corollary. 
\begin{coro}\label{coroEdyba}
Under  Assumptions~\ref{assm_prob}--\ref{assm_stepsize}, w.p.1., 
	$$
	\lim_{t\to\infty} \big( \dist^2_{\mathcal{G}}(\yb(t)) + \epsilon_b b_{t} + \epsilon_a a_{t} \big) = 0. 
	$$
\end{coro}
\begin{proof}
    For sufficiently large $t$, Theorem~\ref{thmEdyba} tells us
    \begin{align}
	&\Es{t}{\dist^2_{\mathcal{G}}(\yb(t+1)) + \epsilon_b b_{t+1} + \epsilon_a a_{t+1}} \nnb\\
	& \le 
	\epsilon \big( \dist^2_{\mathcal{G}}(\yb(t))  +  \epsilon_b b_t  + \epsilon_a a_t \big)  +  \mathcal{O}(p_t). \nnb
	\end{align}
    Note that, for any $T\in \N$, 
    \begin{align}
    \textstyle \sum_{t=0}^T p_t 
    &= 
    \textstyle\sum_{t=0}^T  \big(\gamma^2_{t} + 
	\sigma^{2(t+1)} + \textstyle\sum_{s=0}^t \sigma^{{t- s}} \gamma_s^2 \big) \nnb\\
	&\le 
	\textstyle \frac{\sigma^2}{1-\sigma^2} + \sum_{t=0}^T \gamma^2_{t} + \sum_{t=0}^T \sum_{s=0}^t \sigma^{t- s} \gamma_s^2 \nnb\\
	&= 
	\textstyle \frac{\sigma^2}{1-\sigma^2} + \sum_{t=0}^T \gamma^2_t + \sum_{s=0}^T \gamma_s^2 \sum_{k=0}^{T-s} \sigma^k \nnb\\
	&\le 
	\textstyle \frac{\sigma^2}{1-\sigma^2} + \frac{2-\sigma}{1-\sigma}\sum_{t=0}^T \gamma^2_t. \label{eqSumPt}
    \end{align}
    Thus, under Assumption~\ref{assm_stepsize}, $ \sum_{t\in\N} p_t < \infty$. 
    The convergence of corollary now follows from  Lemma~\ref{lemConvergenceRandomSeq}. 
\end{proof}

We are now ready to present the main convergence result of the paper, which will be proved using Lemma~\ref{lemConvergenceRandomSeq2} with
	\begin{align*}
	v_t = \dist^2_{\mathcal{Y}^*}(\yb(t))+ \dist^2_{\mathcal{G}}(\yb(t)) + \epsilon_b b_{t} + \epsilon_a a_{t}.
	\end{align*}
We also consider the following running average terms:
	\begin{align}
	\tilde{\yb}(s,t) \!=\! \frac{\sum_{k=s}^t\gamma_k\yb(k)}{\sum_{k=s}^t\gamma_k}, \ 
	\tilde{\bx}(s,t) \!=\! \frac{\sum_{k=s}^t\gamma_kP_{\mathcal{G}}\big(\yb(k)\big)}{\sum_{k=s}^t\gamma_k} \label{eqX_Y_tilde}
	\end{align}
\begin{theorem} \label{thm:main_convg}
	Suppose that Assumptions~\ref{assm_prob}--\ref{assm_graph} hold and step sizes satisfy $\lim_{t \to \infty} \gamma_t = 0$. 
	Let $\bar{\gamma} = \sup_t \gamma_t$. 
	\begin{itemize}
		\item[(i)] If  $\E{\dist^2_{\mathcal{Y}^*}(\yb(t))}$ is bounded for all $t\in \N$, then
		\begin{align} 
		&\E{\Phi\big( \tilde\bx(s,t) \big)} \!-\! \Phi^* \!+\! \frac{\rho}{\bar{\gamma}} \E{\|\tilde\yb(s,t) - \tilde\bx(s,t)\|^2} \le E_{s,t} \nnb\\
		&~\text{with}~ E_{s,t} =  \frac{\E{v_s}  +  \mathcal{O}(\sum_{k=s}^t p_k)}{2\sum_{k=s}^t \gamma_{k}}, \quad 
		    0 \leq s \leq t.
		\end{align}
		Moreover, $\lim_{t\to\infty} E_{s,t} \!=\! 0$ for all $s\!\in\!\N$ if $\sum_{t \in \N} \gamma_{t} \!=\! \infty$. 
		
		\item[(ii)] If Assumption \ref{assm_stepsize} holds, then $\dist_{\mathcal{Y}^*}(\yb(t)) \!\to\! 0$ w.p.1.
	\end{itemize}
\end{theorem}
\begin{proof}
	First, we make use of Theorem~\ref{thm_Eyt1w} with $\wb = P_{\mathcal{Y}^*}(\yb(t))$ for each $t$. Under Assumptions~\ref{assm_prob}--\ref{assm_graph}, 
	\begin{align}
	& \Es{t}{\dist^2_{\mathcal{Y}^*}(\yb(t+1))} \nnb\\ 
	&\le\! (1 \!+\! \kappa_t K) \dist^2_{\mathcal{Y}^*}(\yb(t)) \!-\! u_t \!+\! \gamma^2_t\mu^2L^2 \kappa_t^{-1} a_t 
	\!+\! \gamma^2_{t} (D_1 \!+\! a_tD_2), \nnb
	\end{align}
	where 
	\begin{equation*}
	u_t:= 2\gamma_{t}\big[\Phi\big(P_{\mathcal{G}}(\yb(t)) \big) - \Phi^*\big]\!+\! 2\rho \dist^2_{\mathcal{G}}(\yb(t)) \ge 0. 
	\end{equation*}
	Taking $\kappa_t = \gamma^2_t\mu^2L^2/\epsilon\epsilon_a$, we can write the above relation as
	\begin{align}
	& \Es{t}{\dist^2_{\mathcal{Y}^*}(\yb(t+1))} \nnb\\ 
	&\le \big(1+ \kappa_t K \big) \dist^2_{\mathcal{Y}^*}(\yb(t)) -u_t + \epsilon \epsilon_a a_t  
	+ \gamma^2_{t} (D_1 + a_t D_2).\nnb
	\end{align}
	Adding this relation with~\eqref{eqEdyba}, we obtain 
	\begin{align} 
	&\Es{t}{\dist^2_{\mathcal{Y}^*}(\yb(t+1))+ \dist^2_{\mathcal{G}}(\yb(t+1)) + \epsilon_b b_{t+1} + \epsilon_a a_{t+1}} \nnb\\
	& \le 
	\big(1+ \kappa_t K \big) \dist^2_{\mathcal{Y}^*}(\yb(t))+ \epsilon \big( \dist^2_{\mathcal{G}}(\yb(t))  +  \epsilon_b b_t + \epsilon_a a_t\big) \nnb\\
	&\quad - u_t + \gamma_t^2a_t(D_2+D_3+\epsilon_bD_3)   +  \mathcal{O}(p_t). \label{eqEd2yba}
	\end{align}
	
    Under the assumption that $\E{ \dist^2_{\mathcal{Y}^*}(\yb(t))}$ is bounded for all $t \in \N$, i.e., 
    $\E{\dist^2_{\mathcal{Y}^*}(\yb(t))} = \mathcal{O}(1)$, 
    taking the expectation of~\eqref{eqEd2yba}, we obtain the following for all sufficiently large $t$.
	\begin{align} 
	&\E{\dist^2_{\mathcal{Y}^*}(\yb(t+1))+ \dist^2_{\mathcal{G}}(\yb(t+1)) + \epsilon_b b_{t+1} + \epsilon_a a_{t+1}} \nnb\\
	& \le 
	\E{\dist^2_{\mathcal{Y}^*}(\yb(t))+ \epsilon \big( \dist^2_{\mathcal{G}}(\yb(t))  +  \epsilon_b b_t + \epsilon_a a_t\big)} - \E{u_t}\nnb\\
	&\quad + \mathcal{O}(\gamma_t^2)\E{a_t}   +  \mathcal{O}(p_t).\nnb\\
	&\le 
	\E{\dist^2_{\mathcal{Y}^*}(\yb(t))+  \dist^2_{\mathcal{G}}(\yb(t))  +  \epsilon_b b_t + \epsilon_a a_t } - \E{u_t}  +  \mathcal{O}(p_t),  \nnb 
	\end{align}
	where we have used $\kappa_tK\E{\dist^2_{\mathcal{Y}^*}(\yb(t))} + \mathcal{O}(p_t) = \mathcal{O}(p_t)$ 
	in the first inequality and 
	$\big(\epsilon \epsilon_a+ \mathcal{O}(\gamma_t^2)\big) \E{a_t}  \le \epsilon_a\E{a_t}$ 
	in the second. 
	After rearranging terms of the above relation, 
	we have
	\begin{align} 
	\E{u_t} \le \E{v_t} - \E{v_{t+1}}  +  \mathcal{O}(p_t) \nnb
	\end{align}
	which, together with the fact that $\E{v_{t+1}} \ge 0$, implies 
	\begin{align} 
	\textstyle \sum_{k=s}^t \E{u_k} \le \E{v_s}  +  \mathcal{O}(\sum_{k=s}^tp_t).\label{eqSumUst}
	\end{align}
	Next we bound the left-hand side. Note that 
	\begin{align*}
	&\textstyle  \sum_{k=s}^t u_k 
	= \sum_{k=s}^t \big( 2\gamma_{k}\big[\Phi\big(P_{\mathcal{G}}(\yb(k)) \big) - \Phi^*\big] +  2\rho \dist^2_{\mathcal{G}}(\yb(k)) \big)\\
	&\ge 2 \textstyle \sum_{k=s}^t \gamma_{k}\big[\Phi\big(P_{\mathcal{G}}(\yb(k)) \big) - \Phi^*  + \frac{\rho}{\bar{\gamma}} \dist^2_{\mathcal{G}}(\yb(k)) \big]\\
	&\ge 2 \textstyle \big(\sum_{k=s}^t \gamma_{k}\big) \big[\Phi\big(\tilde\bx(s,t) \big) - \Phi^* + \frac{\rho}{\bar{\gamma}} \|\tilde\yb(s,t) - \tilde\bx(s,t)\|^2 \big], 
	\end{align*}
	where we used convexity of $\Phi(\cdot)$ and $\|\cdot\|^2$ and the definitions of $\tilde\yb(s,t)$ and $\tilde\bx(s,t)$ in~\eqref{eqX_Y_tilde}. Using this bound for~\eqref{eqSumUst}, we have 
	\begin{align} 
	&\E{\Phi\big(\tilde\bx(s,t) \big)} - \Phi^* + \frac{\rho}{\bar{\gamma}} \E{\|\tilde\yb(s,t) - \tilde\bx(s,t)\|^2} \nnb\\
	&\le \frac{\E{v_s}  +  \mathcal{O}(\sum_{k=s}^t p_k)}{2\sum_{k=s}^t \gamma_{k}}. \nnb
	\end{align}
	This proves statement (i) of the theorem. 
	
	To show statement (ii), note from \eqref{eqEd2yba} that 
	\begin{align*}
	\Es{t}{v_{t+1}} \le (1+\alpha_t)v_t - u_t + \mathcal{O}(p_t), 
	\end{align*}
	where $\alpha_t = \kappa_t K = \gamma^2_t\mu^2L^2K/\epsilon\epsilon_a$. 
	Applying Lemma~\ref{lemConvergenceRandomSeq2} to this relation and noting that $\sum_t p_t < \infty$ (see \eqref{eqSumPt}), the  following holds w.p.1: 
    \bitem 
	\item[(a)] $\sum_{t \in \N} \gamma_t(\Phi\big(P_{\mathcal{G}}(\yb(t)) \big) - \Phi^*) < \infty$;
	
	\item[(b)] $\sum_{t \in \N} \dist^2_{\mathcal{G}}(\yb(t)) < 
	\infty$; and 
	
	\item[(c)] there is some nonegative RV $v$ such that 
    $v_t \to v$.
	\eitem 
	Since $\sum_{t \in \N} \gamma_t = \infty$ and 
	$\dist^2_{\mathcal{G}}(\yb(t)) + \epsilon_b b_{t} + \epsilon_a a_{t} \to 0$ (Corollary~\ref{coroEdyba}), the findings (a) and 
	(c) imply $\liminf_{t} \Phi\big(P_{\mathcal{G}}(\yb(t))$  
	$= \Phi^*$ and 
	$\dist^2_{\mathcal{Y}^*}(\yb(t)) \to v$ as 
	$t \to \infty$.
	Because $\Phi$ is continuous, it follows that $v= 0$ w.p.1. This proves the second part of the theorem. 
\end{proof}
	
\begin{remark}
{\rm Note that the boundedness condition of $\E{\dist^2_{\mathcal{Y}^*}(\yb(t))}$ in Theorem
\ref{thm:main_convg}(i) is satisfied if 
$\mathcal{G}_{i0}$, $i \in \cA$, are compact, 
which will likely hold in many, if not most, cases
of practical interest.
Also, the convergence rate of the algorithm 
obviously depends on the choice of step sizes. For 
instance, if $\gamma_t = \mathcal{O}(\frac{1}{\sqrt{t}})$, 
then $E_{\lfloor t/2 \rfloor,t} = \mathcal{O}(\frac{1}{\sqrt{t} })$ \cite{beck2017first}.}
\end{remark}

\section{Conclusion}

We studied solving a constrained 
optimization problem using noisy observations, 
and proposed a new distributed algorithm that 
does not require sharing optimization variables
among agents. Instead, the agents update their local estimates of 
global constraints using a consensus-type
algorithm, while updating their own local 
optimization variable based on noisy estimates
of gradients of local objective functions. 
We proved that (a) the optimization 
variables converge to an optimal point of an approximated
problem and (b) the tracking errors of local 
estimates of constraint functions vanish 
asymptotically.

    
    \bibliographystyle{plain}
    \bibliography{TAC_TN_arXiv}


\end{document}